\documentclass[12pt]{amsart}
\usepackage{graphicx}

\usepackage{amsmath, amsthm, amscd, amsfonts,amssymb}

\newif\ifpdf
\newtheorem{theorem}{Theorem}[section]
\newtheorem{lemma}[theorem]{Lemma}
\newtheorem{proposition}[theorem]{Proposition}
\newtheorem{corollary}[theorem]{Corollary}
\newtheorem{observation}[theorem]{Observation}
\theoremstyle{definition}
\newtheorem{definition}[theorem]{Definition}
\newtheorem{example}[theorem]{Example}

\theoremstyle{remark}

\numberwithin{equation}{section}

\newfont{\kh}{msbm10}

\begin{document}
\title[Some results in a new power graphs in finite groups  ]
{Some results in a new power graphs in finite groups }
\author{S. H. Jafari}
\address{S. H. Jafari, \newline Department of Mathematics,
University of Shahrood , P. O. Box 3619995161-316,
Shahrood, Iran} \email{shjafari55@gmail.com }


\begin{abstract}
The power graph of a group is the graph whose vertex set is the
set of non-trivial elements of group, two elements being adjacent
if one is a power of the other.  We define a new power graph and
study on connectivity, diameter and clique number of this graphs.

\end{abstract} \maketitle

\section{Introduction.}
 The directed power graph of a
semigroup $S$ was defined by Kelarev and Quinn \cite{kq} as the
digraph $\Gamma(S)$ with vertex set S, in which there is an arc from
$x$ to $y$ if and only if $x\neq y$ and $y=x^m$ for some positive
integer $m$. Motivated by this, Chakrabarty et al. \cite{css}
defined the (undirected) power graph $\Gamma(S)$, in which distinct
$x$ and $y$ are joined if one is a power of the other.

Let $L$ be a graph. We denote $V(L)$  and $E(L)$ for vertices and
edges of $L$, respectively.  For any two vertex $x,y$ of $L$, we
denote $d(x,y)$ for the length of smallest path between $x,y$.
Also $diam(L)=Max\{d(x,y)| x,y\in V(L)\}$. The clique number of
$L$, is denoted by  $w(L)$, is the maximum size  of  complete
subgraphs of $L$.
 The (open) neighborhood $N(a)$ of vertex $a\in V(L)$ is the set of
vertices are adjacent to $a$. Also  the closed neighborhood of
$a$, $N[a]$ is $N(a)\cup \{a\}$. All groups in this paper are
finite.

Since in current power graphs the trivial element is adjacent two
any element and this property is not  very good, we define a new
power graph.

 \begin{definition} Let $G$ be a group and $S$ is a subset of
$G$. We define the power graph $\Gamma(G,S)$ with vertex set $S$
and two elements are adjacent if and if only one of them is power
of another. Also we denote $\Gamma_1(G)$ for $\Gamma(G,S)$ where
$S=G-\{1\}$.
 \end{definition}.

\begin{observation} \label{ob}
 In the power graph of a finite group $G$,

 (i) $N[a]=N[a^i]$ for any $a \in G$
and $(o(a),i)=1$.

(ii) $\Gamma(G)$ is complete if and only if $G$ is a cyclic group of prime power order.
\end{observation}

\section{connectivity in power graphs}


\begin{lemma} \label{con1}
Let $G$ be a finite group. Then $\Gamma_1(G)$ is connected if and
only if for any two  elements $x,y$ of prime orders where $\langle
x\rangle \neq \langle y \rangle$, there exist elements
$x=x_0,x_1,\ldots,x_t=y$  such that $o(x_{2i})$ is prime,
$o(x_{2i+1})=o(x_{2i})o(x_{2i+2})$ for $i \in \{0,...,t/2\}$ and,
$x_i$ is adjacent to $x_{i+1}$ for $i \in \{0,1,...,t-1\}$.
\end{lemma}

\begin{proof}
"$\Longleftarrow$" Is trivial.

 "$\Longrightarrow$"
Assume that $\Gamma_1(G)$ is connected and $x,y$ are two elements
of prime orders where $d(x,y)$ is minimum and $x,y$ do not satisfy
in conclusion of Lemma. Let $x=x_0-x_1-...-x_t=y$ is a smallest
path. Consequently $o(x_i) \neq o(x_{i+1})$. Since $o(x)$ is
prime then $x_0=x_1^s$, for some integer $s$. Thus $o(x)<o(x_1)$
and $o(x)\mid o(x_1)$. Assume that $k$ is a greatest integer such
that $o(x_0)< o(x_1)< \cdots < o(x_k)$. Also let $m$ is a greatest
integer such that $o(x_k)>o(x_{k+1}> \cdots >o(x_m)$. We have
$\langle  x_0 \rangle \subseteq \langle x_1\rangle \subseteq ...
\subseteq \langle x_k\rangle$ and
 $\langle  x_m \rangle
\subseteq \langle x_{m-1}\rangle \subseteq ... \subseteq \langle
x_k\rangle$. Then $x_0,x_m  $ are adjacent to $x_k$ and $k=1,
m=2$. Since  $x_0,x_2 \in \langle x_1 \rangle$ and $x_0,x_2$ are
not adjacent, then $\langle x_2 \rangle$ has element $b$ of prime
order distinct to $o(x_0)$. Now $x_0-bx_0-b$ satisfies in Lemma
and  $x=x_0-bx_0-b-x_3-...-x_t=y$  is a path. Consequently
$x_3-...-x_t$ do not satisfy Lemma, a contradiction.
\end{proof}
For $n\geq 3$ we consider the generalized Quaternion group as
following.

 $Q_n=\langle a,b|a^n=b^2, b^{-1}aba=b^4=1 \rangle$.
 \begin{observation} \label{obq} \cite[Th 5.4.10]{go}
 The finite $p$-group has exactly one subgroup of order $p$ if and
 only if it is a cyclic or  a generalized Quaternion group.
  \end{observation}
\begin{corollary}\label{cor1}
Let $G$ be a $p$-group. Then $\Gamma_1(G)$ is connected if and
only if $G$ is  a cyclic or a generalized Quaternion group.
\end{corollary}
\begin{proof}
Let $\Gamma_1(G)$ is connected. Since any two element $x,y$ of
order $p$ is adjacent if and only if $\langle x \rangle=\langle y
\rangle$, then by Lemma \ref{con1}, $G$ has exactly one subgroup
of order $p$. Now the proof is completed, by Observation
\ref{obq}.
\end{proof}
\begin{corollary}
For $p$-group $G$, the number of connectivity complement of
$\Gamma_1(G)$ is equal to the number of subgroups of order $p$.
\end{corollary}
\begin{proposition}\label{prop}
If $G,H$ is two nontrivial groups and $(|G|,|H|)=1$, then
$\Gamma_1(G \times H)$ is connected.
\end{proposition}
\begin{proof}
 Let $p,q$ be  two prime factors of
   $|G \times H|$ and $x,y$ be of order of
$p,q$, respectively. We have tree cases.

     Case 1.
      $x,y \in G \times \{1\}$. Then, there exists an element

      $z$ of prime order in $H$. Now  $x-x(1,z)-(1,z)-y(1,z)-y$ is a

      path.

     Case 2.
      $x,y \in \{1\} \times H$. Is similar to Case1.

     Case 3.
      $x \in G \times \{1\}$ and $y \in \{1\} \times H$. Now

      $x-xy-y$ is a path.

Now the proof is complete by Lemma \ref{con1}.
\end{proof}
\begin{proposition}
Let $G,H$ be two groups where $\Gamma_1(G)$ is connected
 and $|G|$ is not prime power.   Then $\Gamma_1(G \times H)$ is connected.
 \end{proposition}
 \begin{proof}
 Let $(a,b)$ be an element of prime order in $G \times H$ and $b\neq 1$. We set $o(a,b)=o(b)=p$.
 By Lemma \ref{con1}, there exist $x \in G$ such that $o(x)$ is prime, $(o(x),p)=1$ and $ax=xa$.
 Now $(a,b)-(xa,b)-(a,1)$ is a path. Consequently every element of prime order is connected to
 an element of prime order in $G \times 1$.

 By Lemma \ref{con1}, the proof is  complete.
  \end{proof}
\begin{theorem}\label{sc}
Let $G$ be a finite nilpotent group. Then $\Gamma_1(G)$ is
connected if and only if $G$ is cyclic or generalized Quaternion
or not a $p$-group for any prime $p$.
\end{theorem}
\begin{proof}
By Corollary \ref{cor1} we assume that $|G|$ is not a prime
power. Thus  $G=P_1 \times P_2 \times ... \times P_t$ where
$P_1,...,P_t$ are all syllow subgroups of $G$. Consequently by
Proposition \ref{prop}, $\Gamma_1(G)$ is connected.
\end{proof}
\begin{example}
$\Gamma_1(S_3 \times S_3)$ is not connected but $\Gamma_1(S_3
\times Z_6)$ is
 connected where $S_3$ is   the Symmetric group of degree $3$ and
  $Z_6$ is  the cyclic group of order $6$.
 \end{example}

\section{diameter and clique number of this new power graphs of finite groups}

Now we calculate the diameter of this graph. We note
 that the diameter of old power graph is $1$ or $2$ trivially.
\begin{observation} \label{ob}
Let $G$ be a group. Then $diam(\Gamma_1(G))=1$ if and only if $G$
is a cyclic $p$-group.
\end{observation}
\begin{proof}
In this case $\Gamma_1(G)$ is complete and the proof is concluded  by Observation \ref{ob}.
\end{proof}
\begin{proposition}\label{propd2}
Let $G$ be a group. Then $diam(\Gamma_1(G))=2$ if and only if one
of the following is holds.

 (i) $G$ is a cyclic group of not prime  power order.

  (ii)$G$ is  a direct product of a cyclic group of odd order to
  a generalized  Quaternion group of order of  power of $2$.
 \end{proposition}
 \begin{proof}
 " $\Longrightarrow$"
 We first prove that $G$ is nilpotent. Let $x,y$ are two
 elements of order $p^n,q^m$ where $p,q$ are
 two distinct primes. Let $x-a-y$ is a path.
 If $a \in \langle x \rangle$
 then $ o(a)|o(x)=p^n$. Since $(p^n,q^m)=1$,
 we conclude that $a$ is not adjacent to $y$, a contradiction.
 Thus $x \in\langle a \rangle$ and
 respectively $y \in\langle a \rangle$.
 Consequently $xy=yx$, and then $G$ is nilpotent. Therefore
$G=P_1 \times ... \times P_t$, where $P_1, ..., P_t$ are all
syllow subgroups of $G$. Assume that  $x,y$ are two elements of
order $p_i$ in $P_i$ and $\langle x\rangle \neq \langle y\rangle$.
Let $x-a-y$ is a path. By similarly the above we should have
$x,y\in \langle a\rangle$, a contradiction. Thus $P_i$ has
exactly one subgroup of prime order, and consequently is a cyclic
or generalized Quaternion group. Now by  the Observation
\ref{ob}, the poof is complete.

 " $\Longleftarrow$"
 Let $G$ is satisfying in one of  those condition. We note that $G$ has exactly one
  subgroup of any prime factors of $|G|$. Let $G$ be a cyclic group of order of not prime power
 and $x,y \in G$. Let $x,y$ are not adjacent. If $(o(x),o(y))=1$
 then $o(xy)=o(x)o(y)$ and $x,y \in \langle xy \rangle$.
 Therefore $x-xy-y$ is a path. We assume that $p|(o(x),o(y))$ and $a$ is an element
 of order $p$  in $\langle x \rangle$. $a\in \langle y \rangle$, because $G$ is cyclic.
  Thus $x-a-y$ is a path.
 Since $\Gamma_1(G)$ is not complete, $diam(G)=2$.

   (ii) is same to (i).
   \end{proof}


\begin {theorem}\label{t1}
Let $G$  be a nilpotent group that $diam(G)>2$ then $diam(G)=4$.
\end{theorem}
\begin{proof}
We have  $G=P_1 \times P_2 \times ... \times P_t$ where $P_i$ is
$p_i$-syllow subgroup of $G$ for $i\in \{1,...,t\}$. Also $G$ is
not cyclic of prime power. By Proposition \ref{propd2}, $G$ has
two cyclic subgroup of prime order $p_i$ for some $i$. Let $x,y$
be two element of order $p_i$ where $\langle x \rangle \neq
\langle y\rangle$. Indeed $d(x,y)>2$. Let $d(x,y)=3$ and
$x-a-b-y$ is a path. we see that $x\in \langle a\rangle$ and
$y\in \langle b \rangle$. If $a\in \langle b\rangle$ then
$x,y\in\langle b\rangle$, a contradiction.
 Similarly $b\notin \langle a\rangle$ and so $a$ is not adjacent to
 $b$, a contradiction. Consequently $d(x,y) \geq4$ and
  then $diam(\Gamma_1(G))\geq 4$. On the other
hand, for any two elements $x,y \in G$  we have two cases:

 Case1. $x,y \in P_i$ for some $i$. We have $x-xa-a-ay-y$ is a
       path where $a \in P_j$ and $j\neq i$.

       Case2. $x,y \notin P_i$, for each $i$. Then there exist primes integer
        $p,q$, such that $p|o(x)$,$q|o(y)|$ and $p\neq q$.
         Let $a\in \langle x\rangle$  of order $p$ and
 $b\in \langle y\rangle$ of order $q$. We see
 that $x-a-ab-b-y$ is a path.

  Consequently $diam(\Gamma_1(G))=4$.
\end{proof}

Note that by the Observation \ref{ob}, Proposition \ref{propd2} and
Theorem \ref{t1}, the diameters of all finite nilpotent groups
were calculated.

\begin{example}

 (i)$diam(Q_n)=3$ for odd integer  $n$.

 (ii)$diam(S_3 \times Z_6)=4$.
\end{example}

Now we calculate the clique number of power graph of finite
groups.

 Let $n$ be an integer. The set $A=\{d_1,..d_t\}$  of
divisors of $n$ is said a $CD$-set  of $n$ if $d_1>1$ and one of
any two elements of $A$ divided by another. The maximal of these
set is said an $MCD$-set of $n$. We define
$weight(A)=\varphi(d_1)+...+\varphi(d_t)$ and

$weight(n)=max\{ weight(A)| A$ is an $MCD$-set of $n$\}.

\begin{lemma}\label{mcs}
Let $A=\{d_1,...,d_t\}$, where $d_1<...<d_t$. Then $A$ is an
$MCD$-set if and only if $d_1$ be a prime, $d_t=n$ and,
$d_{i+1}/d_i$  is a prime for $i\in\{2,..,t\}$
\end{lemma}
\begin{lemma}\label{mcs1}
Let $G$ be a group and $K$ is an maximal complete subgraph. Then
there exist an maximal cyclic subgroup $\langle a \rangle$ such
that $V(K) \subseteq \langle a \rangle$. Moreover  $A=\{o(x)|x
\in V(K)\}$ is an $MCD$-set of $o(a)$.
\end{lemma}
\begin{proof}
Assume that $a\in V(K)$ and $o(a)=max(A)$.
 Clearly $V(K)\subseteq \langle a \rangle$, $\langle a \rangle$
 is an maximal cyclic subgroup of $G$ and $A$ is a $CD$-set of $o(a)$.
 Let $A$ is not an  $MCD$-set and $A=\{o(a_1),... , o(a_t)\}$,
 where $o(a_1)\leq ... \leq o(a_t)$.
 Then,   there exist $i$ such that $o(a_i)\neq O(a_{i+1})$ and
 $o(a_{i+1})/o(a_i)$  is not a prime. Assume that $p|o(a_{i+1})/o(a_i)$ and
 $b$ be an element of order $po(a_i)$ in $\langle a_{i+1} \rangle$.
  We see that $V(K) \subseteq N(b)$, a contradiction.
 \end{proof}

\begin{theorem}
Let $G$ be a finite group. Then
$\omega(\Gamma_1(G))=\omega(\Gamma(G))-1=max\{weight(o(a))| a\in
G\}$.
\end{theorem}
\begin{proof}
Assume that $K,A$ and $a$ be similar  to last Lemma. We have
$\langle a_i \rangle$ and so $V(K)$ has exactly $\varphi(o(a_i))$
elements of order $o(a_i)$. By  Lemma \ref{mcs1}, the proof is
complete.
\end{proof}
\begin{corollary}
For the nilpotent group $G$,

 $\omega(\Gamma_1(G))=weight(exp(G))$

where $exp(G)=Min\{ n|x^n=1$ for any $x\in G\}$.
\end{corollary}


\end{document}